\documentclass[12pt]{article}
\usepackage{amsmath}
\usepackage{amssymb}
\usepackage{amstext}
\usepackage{amsthm}
\usepackage{geometry}

\newtheorem{Theorem}{Theorem}%[section]
\newtheorem*{Theorem A}{Theorem A}
\newtheorem{Lemma}{Lemma}

\newtheorem{Remark}{Remark}

\begin{document}

\title{An answer to Hammerlindl's question on strong unstable foliations}

\author{Ren Yan\thanks{School of Mathematical Sciences,
Peking University, Beijing, China 100871. Email: renwyx379@163.com.}
 \and
Shaobo Gan\thanks{School of Mathematical Sciences,
Peking University, Beijing, China 100871. Email: gansb@pku.edu.cn.}
\and
Pengfei Zhang \thanks{Department of Mathematics and Statistics,
Universityof Massachusetts, Amherst MA 01003, USA. Email: pzhang@math.umass.edu. }}

%\date{\today}

\maketitle

\begin{abstract}
Let $f: \mathbb{T}^3\to\mathbb{T}^3$ be a partially hyperbolic diffeomorphism on the 3-torus $\mathbb{T}^3$.
In his thesis, Hammerlindl proved that for lifted center foliation $\mathcal{F}^c_f$,
there exists $R>0$, such that for any $x\in \mathbb{R}^3$, ${\cal
F}^c_f(x)\subset B_R (x+E^c)$, where $\mathbb{R}^3=E^s\oplus E^c\oplus
E^u$ is the partially hyperbolic splitting of the linear model of $f$.
The same is true for the lifted center-stable and center-unstable foliations.
Then he asked if the this property is true for strong stable and strong unstable foliations.
In this note, we give a
negative answer to Hammerlindl's question.
\end{abstract}

%\section{Introduction}

Let $M$ be a compact $C^\infty$ Riemannian manifold without boundary. A
diffeomorphism $f: M\to M$ is said to be a partially hyperbolic diffeomorphism (abbrev.
PHD), if there exists $Tf$-invariant splitting $TM=E^s\oplus E^c\oplus E^u$
and constants $C\ge 1$ and
\begin{equation}\label{abs}
0<\lambda<\hat{\gamma}<1<\gamma<\mu,
\end{equation}
such that for any $x\in M, n\ge 0$,
\begin{equation}\label{phd}
\begin{array}{cccl}\displaystyle
& |Tf^n v|&\le C\lambda^n|v|, &\forall\, v\in E^s(x),\\
C^{-1}\hat{\gamma}^n|v|\le &|Tf^n v|&\le C\gamma^n|v|, &\forall\, v\in
E^c(x),\\
C^{-1}\mu^n|v|\le  &|Tf^n v|,& &\forall\, v\in E^u(x).
\end{array}
\end{equation}
Denote the joint subbundles by $E^{cs}=E^c\oplus E^s$, $E^{cu}=E^c\oplus E^u$.
We will use $E^\sigma_f$ ($\sigma=s,c,u,cs,cu$) to indicate the dependence
on $f$.

\begin{Remark}
Recently, the definition formulated in \eqref{phd} is called absolutely
partial hyperbolicity, in contrast to another well explored formulation--pointwise partial
hyperbolicity, for which the constants in \eqref{abs} are replaced by continuous functions,
and in \eqref{phd}, by the corresponding cocycles
(see \cite{HP13a} for more details and the reference therein).
We will not touch upon the difference between this two kinds of definitions in this note.
\end{Remark}

E. Pujals had an informal
conjecture about the complete classification of all partially
hyperbolic diffeomorphisms on 3-manifolds (see \cite{BW05}
for some recent development and improvement of the conjecture).
Inspired by Pujals' conjecture, great
progresses have been made very recently towards such a classification
of partially hyperbolic diffeomorphisms on various 3D manifolds.
See \cite{Ham13} for the case $M=\mathbb{T}^3$, \cite{HP13a} for 3D nilmanifolds,
and \cite{HP13b} for 3D manifolds with solvable fundamental groups.

Now let's state the question of Hammerlindl. For simplicity, we will
concentrate on partially hyperbolic diffeomorphisms on 3-torus
$\mathbb{T}^3$ such that $\dim E^\sigma=1, \sigma=s,c,u$. A partially
hyperbolic diffeomorphism $f:\mathbb{T}^3\to\mathbb{T}^3$ is {\it dynamically
coherent}, if there exist two foliations ${\cal F}^{cs}, {\cal F}^{cu}$
tangent to $E^{cs}$, $E^{cu}$ respectively. Let
$F=A+\phi:\mathbb{R}^3\to\mathbb{R}^3$ be a lift of $f$, where
$\phi(x+{\bf n})=\phi(x)$ for any $x\in \mathbb{R}^3$ and ${\bf n}\in
\mathbb{Z}^3$. According to \cite{BBI04}, if a PHD
$f:\mathbb{T}^3\to\mathbb{T}^3$ is dynamically coherent, then
$A:\mathbb{R}^3\to\mathbb{R}^3$ is also a PHD (denote its splitting as
$E^s_A\oplus E^c_A\oplus E^u_A$).

In his remarkable thesis \cite{Ham13},
A. Hammerlindl proved that there exists $R>0$, such
that ${F}_F^\sigma(x)\subset B_R(x+E^\sigma)$, where $\sigma=c,cs,cu$,
${\cal F}_F^\sigma$ are the lifted foliations of ${\cal F}_f^\sigma$,
and $B_R(A)=\{x\in\mathbb{R}^3: d(x, y)\le R \text{ for some } y\in A\}$ for any subset
$A\subset \mathbb{R}^3$.
Then he raised the following  (see Page 15 of Thesis):

\vskip.5cm

\noindent {\bf Question: } Let
$f:\mathbb{T}^3\to\mathbb{T}^3$ be a dynamically coherent PHD. Does there
exist $R>0$, such that ${\cal F}_F^\sigma(x)\subset B_R(x+E^\sigma_A)$,
$\sigma=s,u$?

\begin{Remark}
According to \cite{BBI09}, every PHD on 3-torus is dynamically
coherent.
\end{Remark}

In this note, we will give a negative answer to Hammerlindl's question.
Let $A\in SL(3,\mathbb{Z})$ be an integer matrix with three different positive eigenvalues:
$0<\lambda_1<1<\lambda_2<\lambda_3$
(replace $A$ by $A^{-1}$ if $0<\lambda_1<\lambda_2<1<\lambda_3$).
Then the induced map
on 3-torus, $f_A:\mathbb{T}^3\to \mathbb{T}^3$,
can be viewed either as an Anosov diffeomorphism,
or as a PHD with an expanding center bundle $E^c$.
For concreteness, it is easy to see that
$$
A=\left(
\begin{array}{ccc}
 2 & 1 & 1 \\
 1 & 2 & 0 \\
 1 & 0 & 1
\end{array}
\right)
$$
is such a kind of matrix.

\begin{Theorem}\label{main}
Let $A\in SL(3,\mathbb{Z})$  be an integer matrix with three different positive eigenvalues:
$0<\lambda_1<1<\lambda_2<\lambda_3$.
Then there exists a $C^1$ neighborhood ${\cal U}$ of $f_A$ such that for
any $f\in{\cal U}$, if $f$ is accessible, then for any $R>0$, $x, y\in
\mathbb{R}^3$, and for any lift $F$ of $f$, ${\cal F}_F^u(y)\not\subset
B_R(x+E^u)$.
\end{Theorem}

\begin{Remark}
Recall that a PHD $f$ on $M$ is said to be accessible, if for any points $p,q\in M$,
there exists a piecewise smooth path moving from $p$ to $q$ along stable or unstable leaves of $f$.
Note that accessibility is a $C^1$-open and $C^\infty$-dense
property among PHDs on 3-torus,  see \cite{DW03, Did03, RHRHU08}.
Hence the above theorem answers Hammerlindl's question negatively.
\end{Remark}

\begin{proof}[Proof of Theorem \ref{main}]
It is well known that being an Anosov diffeomorphism is a $C^1$-open property
 (e.g., see \cite{Fra70}). So is being a PHD.
So we first pick a $C^1$-small neighborhood $\cal U$ of $f_A$, such that any $f\in{\cal U}$ is
still Anosov and PHD. Now let $f\in\mathcal{U}$.
To avoid any possible confusion, we list the following facts:
\begin{itemize}
\item viewing $f$ as an Anosov system: let
$\mathcal{W}^s_f(x)$ and $\mathcal{W}^u_f(x)$ be the stable and unstable manifolds
of $f$ at $x$.

\item  viewing $f$ as a PHD: Let $\mathcal{F}^s_f$ and $\mathcal{F}^u_f$
be the strong stable and strong unstable foliations
of $f$.

\item viewing $f$ as a perturbation of a PHD with linear center foliation:
$f$ is dynamically coherent (see \cite{HPS}). Let $\mathcal{F}^{cs}_f$  and
$\mathcal{F}^{cu}_f$, and $\mathcal{F}^{c}_f$
be the center-stable, center-unstable and center foliations.

\item the relation between the different viewpoints:
$\mathcal{F}^s_f=\mathcal{W}^s_f$ and
$\mathcal{F}^{cu}_f=\mathcal{W}^{u}_f$.
\end{itemize}

By the structural stability of Anosov systems,
there exists uniquely a topological conjugacy
$h:\mathbb{T}^3\to\mathbb{T}^3$ close to identity
between $f$ and $f_A$, i.e., $h\circ f=f_A\circ h$.
Denote by $\pi: \mathbb{R}^3\to\mathbb{T}^3$ the standard universal
covering map: i.e., if identifying
$\mathbb{T}^3=\mathbb{R}^3/\mathbb{Z}^3$, $\pi$ is just the quotient map.
Let $F, H:\mathbb{T}^3\to\mathbb{T}^3$ be lifts of $f, h$, i.e., $\pi
F=f\pi$, $\pi H=h\pi$. Moreover, we can pick $F$ and $H$ such that $H\circ F=A\circ H$.

Let $z=H^{-1}(0)$. Then $H\circ F(z)=A\circ H(z)=A(0)=0$ and hence $F(z)=z$.
Since $H$ is a conjugacy between Anosov
systems, $H({\cal F}_F^{cu}(z))={\cal F}_A^{cu}(0)$.

\begin{Lemma}\label{unstable}
If $H({\cal F}_F^u(z))\subset {\cal F}_A^{u}(0)$,
then $H({\cal F}_F^u(x))={\cal F}_A^{u}(Hx)$ for any
$x\in\mathbb{R}^3$.
\end{Lemma}

\begin{proof}[Proof of Lemma \ref{unstable}]
Since $H$ is a homeomorphism, $0\in H({\cal F}_F^u(z))$ and ${\cal
F}_A^{u}(0)$ is $A$-invariant, we have
$H({\cal F}_F^u(z))={\cal F}_A^{u}(0)$.

Denote by ${\cal F}_f^u(p, r)=\{q\in {\cal F}_f^u(p): d_u(q,p)\le r\}$,
where $d_u$ is the metric along ${\cal F}^u_f$. Similarly, we can define
${\cal F}_{f_A}^u(p, r)$, ${\cal F}_F^u(x, r)$, ${\cal F}_A^u(x, r)$.
Note that ${\cal F}_f^u(p, r)$ varies continuously with respect to the point $p$.
Following is a special fact due to the linearity of $A$:

\vskip.3cm

\noindent {\bf Fact. } For any $p\in\mathbb{T}^3$, ${\cal F}_{f_A}^u(p)$
is dense in $\mathbb{T}^3$.

\vskip.3cm

\noindent According to the above fact, there exists a dense subset
$S\triangleq H^{-1}({\cal F}_{f_A}^u(0))$
such that for any $r>0$ and $p\in S$, $h({\cal F}_f^u(p,r))\subset {\cal
F}_{f_A}^u(h(p))$. By the continuity of ${\cal F}^u$, we have that for any
$r>0$ and any $p\in\mathbb{T}^3$, $h({\cal F}_f^u(p,r))\subset {\cal
F}_{f_A}^u(h(p))$, which implies $h({\cal F}_f^u(p))={\cal
F}_{f_A}^u(h(p))$. This finishes the proof of the lemma.
\end{proof}

Now we continue the proof of Theorem \ref{main}.
If the assumption of the lemma were satisfied, combing with the fact that $h({\cal
F}_f^s(p))={\cal F}_{f_A}^s(h(p))$, we see that ${\cal F}(p)=:h^{-1}({\cal
F}_{f_A}^{su}(h(p)))$ forms a foliation such that for any $q\in {\cal
F}(p)$, ${\cal F}_f^u(q),{\cal F}_f^s(q)\subset {\cal F}(p)$. In particular, $f$ can't
accessible.

If $f$ is accessible, then
above argument implies that the hypothesis of Lemma~\ref{unstable} is not
satisfied. Therefore there exists a point $x\in{\cal F}^u_F(z)$ such that
$H(x)\notin E^u_A$.
Since $H(\mathcal{W}^u_F(z))=\mathcal{W}^u_A(0)$, we have $d(H(F^n x),
E^u_A)\to +\infty$ as $n\to+\infty$. This implies that for any $R>0$ and
any $x\in\mathbb{R}^3$, $H({\cal F}^u_F(z))\not\subset B_R(x+E^u_A)$. Since
$d(H, Id)$ is bounded, we conclude that for any $R>0$ and any
$x\in\mathbb{R}^3$, ${\cal F}_F(z)\not\subset B_R(x+E^u_A)$.
\end{proof}

\end{document}